\documentclass[article]{amsart}

\usepackage[english]{babel}
\usepackage[utf8x]{inputenc}
\usepackage{amsmath}
\usepackage{graphicx}
\usepackage{enumerate}
\usepackage{amsmath}
\usepackage{mathrsfs}
\usepackage{mathtools}
\usepackage{biblatex}
\usepackage{tikz-cd}
\usepackage{color}

\usepackage[normalem]{ulem}
\usepackage[colorinlistoftodos]{todonotes}
\usepackage{enumitem}
\newtheorem{theorem}            {Theorem}[section]
\newtheorem{proposition}        [theorem]{Proposition}
\newtheorem{lemma}[theorem]{Lemma}
\newtheorem{definition}[theorem]{Definition}
\newtheorem{remark}{Remark}
\newtheorem{cor}[theorem]{Corollary}
\usepackage{amsthm}
\let\oldproofname=\proofname
\renewcommand{\proofname}{\rm\bf{\oldproofname}}

\title{Szeg\H{o}-Type Limit Theorem in the Drury-Arveson Space}
\author{Arya Gayathri Memana}
\begin{document}
\maketitle
\begin{abstract}
   We state and prove a version of Szeg\H{o}'s first limit theorem for Toeplitz-like operators on the Drury-Arveson space in the unit ball.
\end{abstract}

\section{Introduction}

For the complex unit circle $\mathbb{T}:= \{z \in \mathbb{C}: |z|=1\}$, the Hardy space on $\mathbb{T}$ denoted by $H^2(\mathbb{T})$ is the closed subspace of $L^2(\mathbb{T})$ with vanishing negative Fourier coefficients. The functions  $f_n(\theta) = e^{i n \theta}$ for $n \in \mathbb{N} \cup \{0\}$ form an orthonormal basis for $H^2$.

Let $P$ be the orthogonal projection of $L^2$ onto $H^2$. For $\phi \in L^{\infty}(\mathbb{T})$, the operator 
\begin{align*}
    T_{\phi}(f) = P(\phi f)
\end{align*}
is called a Toeplitz operator with the symbol $\phi$. The matrix of $T_{\phi}$ is given by
 \[
T = \begin{bmatrix}
a_0 & a_{-1} & a_{-2} & \cdots & a_{-n} & \cdots \\
a_1 & a_0 & a_{-1} & \cdots & a_{-(n-1)} & \cdots \\
a_2 & a_1 & a_0 & \cdots & a_{-(n-2)} & \cdots \\
\vdots & \ddots & \ddots & \ddots & \ddots & \\
a_n & a_{n-1} & a_{n-2} & \cdots & a_0 & \cdots \\
\vdots & \ddots & \ddots & \ddots & \ddots & \ddots
\end{bmatrix},
\]
  where $a_j=\hat{\phi}(j)$ is the $j^{th}$ Fourier coefficient of $\phi$.\\
  
An object of interest in the study of Toeplitz operators is the Toeplitz $C^*$-algebra $\mathcal{T}$. It is the $C^*$ algebra generated by the shift operator on the Hardy space. Every element in this $C^*-$algebra is of the form $S= T+K$, where $T$ is a Toeplitz operator and $K$ is a compact operator on $H^2$.\\

Szegő's well-known first limit theorem about Toeplitz operators states the following \cite{MR1544404}: Consider a bounded positive Toeplitz operator $T_{\phi}$ on the Hardy space with the continuous symbol $\phi: \mathbb{T} \rightarrow \mathbb{R}$. Denote by $\lambda_1, \lambda_2,..\lambda_N$, the $N$ eigenvalues of the $N \times N$ truncation $T_N$ of $T$. Then, for a continuous function $f: \mathbb{R} \rightarrow \mathbb{R}$,
\begin{align}\label{eq11}
    \lim_{N \rightarrow \infty} \frac{1}{N} \sum_{i=1}^{N} f(\lambda_i) = \lim_{N \rightarrow \infty} \frac{1}{N} Tr(f(T_N)) =  \int_{\mathbb{T}} f(\phi) dm.
\end{align}
where $m$ is the normalized Lebesgue measure on $\mathbb{T}$ and $Tr$ is the usual trace of a matrix.\\

In particular, putting $f(x) = log(x)$, and exponentiating, we have

\begin{align}\label{eq12}
    \lim_{N \rightarrow \infty} det(T_N)^{\frac{1}{N}} = \exp\Big( \int_{\mathbb{T}} log(\phi) \  dm \Big)
\end{align}

More generally, ($\ref{eq11}$) holds for non-negative Lebesgue-integrable functions. See Section 1.6 in \cite{MR2105088} and Sections 1.3 and 5.4 in \cite{MR2223704} for definitions and Chapter 2 in \cite{MR2105088} for history and generalizations.\\
%For the case where $f$ is a continuous function with compact support on $[m,M]$ we can restate this theorem as:
%\begin{align}\label{eq11}
%\lim_{N \rightarrow \infty} \frac{1}{N} Tr(f(T_N)) = \frac{1}{2 \pi} \int_{-\pi}^{\pi} f(\phi(x)) dx.
%\end{align}

Let $P_N$ be the projection onto the linear span of $\{e^{ik \theta}: k=0,1,\cdots,N\}$ on the Hardy space. Then corresponding to each of the operators $P_N TP_N$ we can define a Borel probability measure $\mu_A^N$ on $\mathbb{R}$ as:
\begin{align} \label{eq15}
\mu_T^N(B) = \frac{\nu_T^N(B)}{rank (P_N)} \ \ \ (B \subseteq \mathbb{R}, B \ \text{Borel})
\end{align}
where $\nu_T^N$ is the number of eigenvalues of $P_N T P_N$ in B, counting multiplicities.\\

The spectrum of a Toeplitz operator is the interval $[a,b]$ where $a$ and $b$ are the min and max values of its symbol $\phi$ respectively. In particular, the spectrum of $P_N T P_N$ which is the support of $\mu_T^N$, is contained in $[a,b]$. Equation (\ref{eq11}) can now be read as: $\lim_{N \rightarrow \infty} \int_{a}^{b} f(x) d\mu_T^N(x) = \int_{a}^{b} f(x) d \mu_*(x)$ where $\mu_*$ is the push-forward of the Lebesgue measure under the map $\phi$. In this case, the support of $\mu_*$ is equal to the spectrum of the operator $T$. Thus, Szeg\H{o}'s theorem also gives a way to estimate the spectrum of a self-adjoint Toeplitz operator from the eigenvalues of its truncations.\\

In this paper, we prove a version of this theorem in the setting of the Drury-Arveson space, a function space on the d-dimensional ball (see Section 2 for definitions). We consider a $C^*$-algebra $\mathcal{T}_d$ analogous to the Toeplitz algebra $\mathcal{T}$. Each operator in this algebra will have a symbol $\phi$, which is a continuous function on the unit sphere. For a natural choice of projections $P_N$ we get the following:
\begin{theorem}\label{eq10}
    Let $T$ be a positive operator in $\mathcal{T}_d$ with a continuous symbol $\phi$ and let $[a,b]$ be the range of $\phi$. Suppose $\lambda_1, \lambda_2,\cdots, \lambda_{d_N}$ are the $d_N$ eigenvalues of $P_N TP_N$ where $d_N= rank(P_N)$. For a continuous function $f$ on $[a,b]$, we have,
    \begin{align*}
    \lim_{N \rightarrow \infty} \frac{1}{d_N} \sum_{i=1}^{d_N} f(\lambda_i) = \lim_{N \rightarrow \infty}\int_a^b f d \mu^{N}_T = \int_{\partial \mathbb{B}_d} f \circ \phi d \sigma
    \end{align*}
    where $\sigma$ is the normalized Lebesgue measure on  $\partial \mathbb{B}_d$.
\end{theorem}
\noindent Here, the measures $\mu_T^N$ will be defined analogously to in $(\ref{eq15})$.
\section{preliminaries}
The Hardy space on the unit disc has a classical generalization called the Drury-Arveson space. This is the space $H_d^2$ of analytic functions on $\mathbb{B}_d$ whose power series coefficents satisy a weighted $l^2$ condition. More precisely,
\begin{align*}
    H^2_d =\{f = \sum_{\alpha \in \mathbb{N}^d} a_{\alpha} z^{\alpha}: ||f||^2 = \sum_{\alpha \in \mathbb{N}^d} |a_{\alpha}|^2 < \infty \}
\end{align*}
The weighted monomials $\{\sqrt{\frac{|w|!}{w!}} z^w\}_w$ form the canonical orthonormal basis for $H_d^2$. For convenience, we denote by $e_w:= \sqrt{\frac{|w|!}{w!}} z^w$. \\

Let $S_i$ for $i \in {1,...,d}$ be the left shift operators on $H_d^2$. That is,
\begin{align*}
    S_i(f) = z_i f
\end{align*}The commutative d-shift is the tuple $S=(S_1,S_2,..,S_d)$. For $\alpha= (\alpha_1, \alpha_2,...,\alpha_d) \in \mathbb{N}^d$ we mean by $S^{\alpha}$ the operator $S_1^{\alpha_1}S_2^{\alpha_2}..S_d^{\alpha_d}$ (i.e., multiplication by $z^{\alpha}$). \\ 

Analogous to the classical Toeplitz algebra $\mathcal{T}$, we have the following:
\begin{definition} (\cite{Hartz2016})
The \textup{Toeplitz algebra} $\mathcal{T}_d$ is defined to be the unital $C^*$-algebra generated by $S$, that is,
\begin{align*}
\mathcal{T}_d = C^* (1,S) \subseteq B(H^2_d)
\end{align*}
\end{definition}
\begin{definition}
    By a \textup{Toeplitz-like operator}, we mean an element in $\mathcal{T}_d$. A \textup{positive Toeplitz-like operator} is a positive element in $\mathcal{T}_d$.
\end{definition}
We have the following theorem due to Arveson \cite{MR1668582}.
\begin{theorem}
Fix $d>0$ and denote the compact operators on $H_d^2$ by $K(H_d^2)$. Then,
   \begin{align}\label{eq4}
\mathcal{T}_d / K(H_d^2) \simeq C(\partial \mathbb{B}_d)
\end{align}
given by sending $S_i$ to $z_i$ on $C(\partial \mathbb{B}_d)$. Thus there exists a short exact sequence,
\[
\begin{tikzcd}
    0\arrow[r] & K(H_d^2) \arrow[r, ""] & \mathcal{T}_d \arrow[r, "\pi"] & C(\partial \mathbb{B}_d)\arrow[r] & 0.
\end{tikzcd}
\]
\end{theorem}
\begin{remark}
    This short exact sequence means that the commutator $[S,S^*]$ is a compact operator since its image is zero in $C(\partial \mathbb{B}_d)$. Thus, finite sums of the form $\sum_{i} p_i(S)^* q_i(S) +K$, where $p_i$ and $q_i$ are polynomials in $S$ and $K$ is a compact operator on $H_d^2$, form a dense set in $\mathcal{T}_d$.
\end{remark}
\begin{definition}
    The image of a Toeplitz operator $T$ under the map $\pi$ in ($\ref{eq4}$) will be called the continuous \textup{symbol} of the Toeplitz-like operator.
\end{definition}
For the projections, we choose the sequence of subspaces $H_N =$ span$\{\sqrt{\frac{|w|!}{w!}} z^w: |w|\leq N\}$ of $H_d^2$ and the corresponding projections $P_N$ onto these subspaces. A simple calculation shows that $rank(P_N) = \sum_{j=0}^{N} {j+d-1 \choose d-1}$. \\

% Given a self-adjoint operator $A$ on a Hilbert space $H$, how must a sequence of projections $P_n$ on $H$ be chosen so that the essential spectrum of $A$ can be computed in terms of eigenvalues of truncations $P_NAP_N$? Many authors like William Arveson (\cite{arveson1994c}) and Erik Bédos (\cite{bedos1996folner}) used $C^*$ algebra machinery to address the general question of eigenvalue distribution theorems for self-adjoint operators on Hilbert spaces. \\
To prove the main theorem, we will make use of ideas of Arveson \cite{MR1276162} and Bédos \cite{MR1458766}, who developed generalizations of the Szeg\H{o} theorem to a more abstract $C^*-$algebra context. Rather than considering a single self-adjoint operator, a concretely presented $C^*$-algebra of operators is considered. In \cite{MR1276162} Arveson proves a generalization for $C^*$-algebras with a unique tracial state. However, this is not the case in our setting. Instead, we use Bédos' approach \cite{MR1458766} for general $C^*$-algebras.\\

Let $H$ be a complex Hilbert space, and let $\{P_N\}$ be a sequence of finite-rank projections in $B(H)$ and $\mathcal{A}\subseteq B(H)$ be a unital $C^*$subalgebra. Let $Tr$ stand for the canonical trace on $B(H)$ and let $||T||_2 = (Tr(T^*T))^{\frac{1}{2}}$ be the Hilbert-Schmidt norm on $B(H)$.\\

A sequence of projections will be called a \textit{Følner sequence} if it satisfies one of the following equivalent conditions.
\begin{lemma}[\cite{MR1458766}, Lemma 1]
    The following conditions are equivalent:
    \begin{enumerate}
        \item $\lim_{N} \frac{||AP_N-P_NA||_2}{||P_N||_2}=0$ for all $A \in \mathcal{A} $ \\
        \item $\lim_{N} \frac{||(I-P_N)AP_N||_2}{||P_N||_2}=0$ for all  $A \in \mathcal{A} $ 
    \end{enumerate}
\end{lemma}

Bédos shows in \cite{MR1458766} that the Hilbert-Schmidt norm above can be equivalently replaced by the trace norm of the operator.\\

 Each state $\chi$ on a $C^*$-algebra $\mathcal{A}$ gives rise to a GNS-triple $(\pi_{\chi}, H_{\chi}, \xi_{\chi})$ associated with $(\mathcal{A}, \chi)$. Denoting $E$ to be the projection-valued spectral measure of the self-adjoint operator $\pi_{\chi}(A)$ we get a Borel probability measure $\mu_A^{\chi}$ defined as:
\begin{align*}
    \mu_A^{\chi}(B) = \langle E(B) \xi_{\chi}, \xi_{\chi} \rangle \ \ \ (B \subseteq \mathbb{R}, B \ \text{Borel})
\end{align*}
A general notion of convergence of spectral distributions in the $C^*-$algebra setting is given by the following:
\begin{definition} [\cite{MR1458766}]
   $\{\{P_N\}, \chi\}$ is a \normalfont{Szeg\H{o} pair} for $\mathcal{A}$ if for all self-adjoint $A \in \mathcal{A}$, $\mu_A^N$ converges weakly to $\mu_A^{\chi}$. In other words, if $A$ is self-adjoint then for all $g \in C(sp(A))$,
\begin{align*}
\lim_{N \rightarrow \infty} \int_{\mathbb{R}} g d \mu_{A}^N = \int_{\mathbb{R}} g d \mu_{A}^{\chi}.
\end{align*} 
\end{definition}

Then a necessary and sufficient condition on $P_N$ and $\chi$ for $\{\{P_N\}, \chi\}$ to be a Szeg\H{o} pair is given by the following theorem. 
\begin{theorem}[\cite{MR1458766}, Theorem 6] \label{eq20}
    The pair $\{\{P_N\}, \chi \}$ is a Szeg\H{o} pair if and only if the following conditions hold:
    \begin{enumerate}[label=(\alph*)]
    \item $\chi$ is the "trace per unit volume w.r.t $\{P_N\}$", i.e, $\chi(A) = \lim_{N \rightarrow \infty} \frac{Tr(P_N A)}{rank(P_N)}$ for all $A \in \mathcal{A}$.
    \item $\{P_N\}$ is a Følner sequence for $\mathcal{A}$.
    \end{enumerate}
\end{theorem}

Szeg\H{o}'s theorem can now be restated as:
Let $\mathbb{T}$ denote the circle group with normalized Haar measure $m$, $P_n$ the orthogonal projection from $L^2(\mathbb{T})$ onto the linear span of $\{\epsilon_k; k =0,1,..,n\}$ in $L^2(\mathbb{T})$ where $\epsilon_k(z) = z^k (z \in \mathbb{T})$ for each $k \in \mathbb{N}$, and $M: L^{\infty}(\mathbb{T}) \rightarrow B(L^2(\mathbb{T}))$ is the representation of $L^{\infty}(\mathbb{T})$ as multiplication operators on $L^2(\mathbb{T})$. Then, $\{\{P_n\}, \bar{m}\}$ is a Szeg\H{o} pair for $M(L^{\infty}(\mathbb{T})), \bar{m}$ denoting the (faithful) tracial state on $M(L^{\infty}(\mathbb{T}))$ obtained by transporting $m$ via $M$.\\

In essence, the normalized eigenvalue distributions of the truncations of the Toeplitz operator converge weakly to the push-forward of the normalized Lebesgue measure on $\mathbb{T}$ under $\phi$.\\

% We use this approach to Szeg\H{o}'s theorem to extend the classical result to operators in the Toeplitz algebra of spaces like Drury Arveson space and the weighted Bergman spaces. \\

%Suppose $T_{\chi}$ is a positive Toeplitz operator and with a polynomial symbol $\chi$. Then by Fejer-Riesz theorem (is this true in Drury Arveson Space) it has to be of the form,
%\begin{equation}
%T_{\chi} = a_0 +\sum_w a_w S^{w} + \sum_w \bar{a_w} {S^w}^*
%\end{equation}

\section{main result}
Our main theorem will follow from a sequence of propositions.
\begin{proposition}\label{eq17}
    $\{P_N\}$ is a Følner sequence for the $C^*$-algebra $\mathcal{T}_d$ generated by the shift operators $S_i$.
\end{proposition}
\begin{proof}
    For notational convenience let us define, for $T \in \mathcal{T}_d$,
    \begin{align*}
    \tau_N(T) : = \frac{||P_N T - T P_N||_2}{||P_N||_2}
    \end{align*}
    We need to show $\lim_{N \rightarrow \infty} \tau_N(T)=0$ for all $T \in \mathcal{T}_d$. \\
    
    Consider the shift operators $S_i, i \in {1,..,d}$. We observe that,
    \begin{align*}
   S_i P_N - P_N S_i= (P_{N+1} -P_N) S_i P_N
    \end{align*}
    So, for orthonormal basis elements $e^{\alpha} = \sqrt{\frac{|\alpha|!}{\alpha!}} z^{\alpha}$, we have,
    \begin{align*}
    || S_i P_N-P_N S_i||_2^2 & = ||(P_{N+1}-P_N) S_i P_N||_2^2\\
    & = \sum_{|\alpha| \leq N} \langle (P_{N+1} - P_N)S_i P_N e^{\alpha}, (P_{N+1}-P_N)S_i P_N e^{\alpha}\rangle\\
    & = \sum_{|\alpha| \leq N} \langle(P_{N+1}-P_N)z_i e^{\alpha}, (P_{N+1}- P_N) z_i e^{\alpha}\rangle\\
    & = \sum_{|\alpha|=N} ||z_i e^{\alpha}||^2\\
    & = \sum_{|\alpha|=N} \frac{|\alpha|!}{\alpha !} \ \ \frac{\alpha_1! \alpha_2! \cdots \alpha_i+1! \cdots \alpha_d !}{(|\alpha|+1)!} \\
    & = \sum_{|\alpha|=N} \frac{\alpha_i+1}{|\alpha|+1}\\
    & = \sum_{|\alpha|=N} \frac{\alpha_i+1}{N+1}\\
    & \leq \sum_{|\alpha|=N} 1\\
    & = \binom{N+d-1}{d-1}
    \end{align*}
    To prove the Følner condition we observe that $||P_N|| = rank(P_N) = {\sum_{j=0}^N \binom{j+d-1}{d-1}}$, so that by above,
    \begin{align*}
    \frac{||P_N S_i - S_i P_N||_2^2}{||P_N||_2^2} \leq \frac{ \binom{N+d-1}{d-1}}{\sum_{j=0}^N \binom{j+d-1}{d-1}}
    \end{align*}
    As the numerator is of the order of $N^{d-1}$ and the denominator is of the order of $N^{d}$ the above ratio goes to zero as $N$ goes to infinity.\\

    The limit is zero again if we replace $S_i$ with $S_i^*$ for each $i \in \{1,2,..,d\}$.\\
    
So now it suffices to show that, 
\begin{enumerate}[label=(\alph*)]
\item for $A, B \in \mathcal{T}_d$, such that $\lim_{N\rightarrow \infty} \tau_N(A)=0$ and $\lim_{N\rightarrow \infty} \tau_N(B)=0$, then $\lim_{N \rightarrow \infty} \tau_N(AB)=0$.
\item for $A, B \in \mathcal{T}_d$ such that $\lim_{N\rightarrow \infty} \tau_N(A)=0$ and $\lim_{N\rightarrow \infty} \tau_N(B)=0$ then, $\lim_{N \rightarrow \infty} \tau_N(A+B)=0$.
\item if $\lim_{N\rightarrow \infty} \tau_N(A)=0$ for $A$ in a dense set in $\mathcal{T}_d$, then the limit is zero for every element in $\mathcal{T}_d$.
\end{enumerate}
Recall that for Hilbert-Schmidt norm we have, $||AB||_2 \leq ||A||_2 ||B||$. To prove $(a)$ we see that,
\begin{align*}
    \tau_N(AB) &=  \frac{||P_N AB - AB P_N||_2}{||P_N||_2}\\
    & = \frac{||P_N AB - AP_N B + AP_NB- AB P_N||_2}{||P_N||_2}\\
    & \leq \frac{||P_N AB - AP_N B||_2}{||P_N||_2} + \frac{||AP_N B - ABP_N ||_2}{||P_N||_2}\\
    &  \leq \tau_N(A) ||B|| + \tau_N(B) ||A||
\end{align*}\\
Thus, $\lim_{N \rightarrow \infty} \tau_N(AB)=0$.\\

To prove $(b)$, note that,
\begin{align*}
\tau_N(A+B) & =  \frac{||P_N (A+B) - (A+B) P_N||_2}{||P_N||_2}\\
& =  \frac{||P_N A+ P_NB - A P_N - B P_N||_2}{||P_N||_2}\\
& \leq \tau_N(A) + \tau_N(B)
\end{align*}
and hence, $\lim_{N \rightarrow \infty} \tau_N(A+B)=0$.\\

Now for $(c)$, let $B \in \mathcal{T}_d$. Given $\epsilon >0$, choose $A$ in the dense set such that $||A-B|| < \epsilon$ (note that the norm here is the operator norm). Then,
\begin{align*}
    \tau_N(B) & = \frac{||P_N B - B P_N||_2}{||P_N||_2}\\
    & = \frac{||P_N (B-A+A) -(B-A+A ) P_N||_2}{||P_N||_2}\\
    & \leq \frac{||P_N(B-A) - (B-A)P_N||_2}{||P_N||_2} + \tau_N(A)\\
    & \leq 2 ||(B-A)|| + \tau_N(A)
\end{align*}
Thus, $\lim_{N \rightarrow \infty} \tau_N(B)=0$. This implies $\{P_N\}$ forms a Følner sequence for $\mathcal{T}_d$.
\end{proof}
\begin{definition}
Define, for $T \in B(H_d^2)$,
\begin{align*}
\chi_N (T) : = \frac{Tr(P_NTP_N)}{rank(P_N)}
\end{align*}
\end{definition}
\begin{remark}
$\chi_N$ is a positive linear functional on $B(H^2_d)$. Also, $\chi_N(I)=1$ and $|Tr(P_NT P_N)| \leq ||T|| rank(P_N)$. Thus, each $\chi_N$ is a state on $B(H)$.  In particular, $||\chi_N||= 1$ for each $N$.
\end{remark}
\begin{proposition}\label{eq19}
For a compact operator $K \in B(H_d^2),  \lim_{N\rightarrow \infty} \chi_N(K)=0$
\end{proposition}
\begin{proof}
We begin by noting that for a finite rank operator $F$, since it is in the trace class of $B(H^2_d)$, $||F||_1 < \infty$ and,
\begin{align}
    \frac{Tr(P_NFP_N)}{rank(P_N)} \leq \frac{||F||_1 ||P||}{||P||_2}  = \frac{||F||_1}{||P||_2} 
\end{align}
and so,
\begin{align}\label{eq3}
    \lim_{N \rightarrow \infty} \chi_N(F)=0.
\end{align}

We also know that finite rank operators are dense in the space of compact operators. Thus, given a compact operator $K$ and an $\epsilon>0$, we choose a finite rank operator $F$ such that $||K-F||< \epsilon$. Now, ($\ref{eq3}$) allows us to choose $N$ big enough so that $|\chi_N(F)|<\epsilon$ for all $N >N'$. So,
\begin{align*}
   | \chi_N(K) | &= |\chi_N(K) - \chi_N(F) +\chi_N(F)|\\ 
   & \leq |\chi_N(K) - \chi_N(F)| + |\chi_N(F)|\\
   & \leq ||\chi_N|| \ ||K-F|| + \epsilon\\
   & \leq 2 \epsilon
\end{align*}
for all $N> N'$ and the desired result follows.
\end{proof}

\begin{proposition}
If $\alpha$ and $\beta$ are multi-indices and $\alpha \neq \beta$ then,
\begin{align*}
\lim_{N\rightarrow \infty} \chi_N({S^{\alpha}}^* S^{\beta})= 0
\end{align*}
\begin{proof}
Note that, for $\alpha \neq \beta$ and each $N\geq 0$,
\begin{align*}
Tr(P_N{S^{\alpha}}^* S^{\beta} P_N) &= \sum_{|w|\leq N } \langle {S^{\alpha}}^* S^{\beta}z^w, z^w\rangle\\
&= \sum_{|w|\leq N}\langle S^{\beta}z^w, S^{\alpha}z^w\rangle\\
& = 0,
\end{align*}
and the result follows.
\end{proof}
\end{proposition}
We require the following variant of the Chu-Vandermonde identity.
\begin{lemma}\label{eq16}
( \cite{MR4020393}, Corollary 5.4) (Chu-Vandermonde-type identity) Let $k_1,k_2,j,w_1, w_2 \geq 0$. Then,
\begin{align*}
\sum_{|w|=j} \binom{k_1+w_1}{w_1} \binom{k_2+w_2}{w_2} = \binom{k_1+k_2+j+1}{j}.
\end{align*}
\end{lemma}
\begin{lemma}
For $w=(w_1,w_2,\cdots, w_d)$ and $\alpha=(k_1,k_2\cdots k_d)$,
\begin{align*}
\sum_{|w|=j} \binom{w_1+k_1}{w_1} \binom{w_2+k_2}{w_2} \cdots \binom{w_d+k_d}{w_d} = \binom{k_1+k_2+\cdots +k_d+ j +(d-1)}{j}
\end{align*}
\end{lemma}
\begin{proof}
By Lemma \ref{eq16}, we know that the result is true for $d=2$. Assuming the result is true for $d-1$, we will show that it is true for $d$. Thus letting $w' =(w_1,w_2,\cdots w_{d-1})$, we have,
\begin{align*}
\sum_{|w'|=l} \binom{w_1+k_1}{w_1} \binom{w_2+k_2}{w_2} \cdots \binom{w_{d-1}+k_{d-1}}{w_{d-1}} = \binom{k_1+k_2+\cdots k_{d-1}+ l +(d-2)}{l}
\end{align*}
Then,
\begin{align*}
&\sum_{|w|=j} \binom{w_1+k_1}{w_1} \binom{w_2+k_2}{w_2} \cdots \binom{w_d+k_d}{w_d}\\
&= \sum_{l=0}^j \  \Bigg(\sum_{|w'|=l}\binom{w_1+k_1}{w_1} \binom{w_2+k_2}{w_2} \cdots \binom{w_{d-1}+k_{d-1}}{w_{d-1}} \Bigg) \binom{j-l+k_d}{j-l}\\
 &=\sum_{l=0}^{j} \binom{k_1+k_2+\cdots +k_{d-1}+ l+ (d-2)}{l}\binom{j-l+k_d}{j-l}\\
 & = \binom{k_1+k_2+\cdots + k_d+ j +(d-1)}{j}
\end{align*}
as desired.
\end{proof}
\begin{proposition}\label{eq18}
    For every multi-index $\alpha= (k_1, k_2,...,k_d)$,
    \begin{align*}
    \lim_{N\rightarrow \infty} \chi_N({S^{\alpha}}^* S^{\alpha})= \frac{(d-1)! \alpha!}{(d-1+|\alpha|)!}
    \end{align*}
\end{proposition}
\begin{proof}
For $w=(w_1,w_2,...,w_d)$ we have $||S^{\alpha  } e_w||^2= \frac{1}{\binom{|\alpha|+|w|}{k_1 + w_1, k_2+w_2,\cdots,k_d+w_d}} \binom{|w|}{w_1,w_2,\cdots,w_d}$. Thus,
\begin{align*}
Tr(P_N {S^{\alpha}}^* S^{\alpha} P_N) &= \sum_{|w|\leq N} ||S^{\alpha} e_w||^2\\
& = \sum_{j=0}^N \sum_{|w|=j} \frac{1}{\binom{|\alpha|+j}{k_1 + w_1, k_2+w_2,\cdots,k_d+w_d}}\binom{j}{w_1,w_2,\cdots,w_d} \\
& = \sum_{j=0}^N \sum_{|w|=j}\frac{(k_1+w_1)!(k_2+w_2)!\cdots(k_d+w_d)!}{(|\alpha|+j)!} \frac{j!}{w_1! w_2! \cdots w_d!}\\
& = \sum_{j=0}^N \frac{j!}{(|\alpha|+j)!} \sum_{|w|=j} \frac{(k_1+w_1)! (k_2+w_2)!\cdots (k_d+w_d)!}{w_1!w_2!\cdots w_d!}\\
& = \sum_{j=0}^N \frac{j!k_1! k_2!\cdots k_d!}{(|\alpha|+j)!} \sum_{|w|=j} \binom{k_1+w_1}{w_1} \binom{k_2+w_2}{w_2}\cdots \binom{k_d+w_d}{w_d}\\
& = \sum_{j=0}^N \frac{j!k_1! k_2!\cdots k_d!}{(|\alpha|+j)!} \binom{|\alpha|+j +(d-1)}{j}\\
& = \sum_{j=0}^N \frac{k_1!k_2!\cdots k_d!}{(|\alpha|+ j)!} \frac{(|\alpha|+j+(d-1))!}{(|\alpha|+(d-1))!}\\
& = \frac{\alpha!}{(|\alpha|+(d-1))!}\sum_{j=0}^N \frac{(|\alpha|+j+(d-1))!}{(|\alpha|+j)!}
\end{align*}
Thus we have,
\begin{align*}
\frac{Tr(P_N{S^{\alpha}}^* S^{\alpha} P_N)}{rank(P_N)}& = \frac{\alpha!}{(|\alpha|+(d-1))!}\frac{\sum_{j=0}^N \frac{(|\alpha|+j+(d-1))!}{(|\alpha|+j)!}}{\sum_{j=0}^N {j+d-1 \choose d-1}}\\
& = \frac{(d-1)!\alpha!}{(|\alpha|+(d-1))!} \frac{\sum_{j=0}^N(|\alpha|+ j+1) (|\alpha|+j+2)\cdots (|\alpha|+j+(d-1))}{\sum_{j=0}^N (j+1)(j+2)\cdots (j+(d-1))}
\end{align*}
Hence,
\begin{align*}
\lim_{N\rightarrow \infty} \frac{Tr(P_N{S^{\alpha}}^* S^{\alpha}P_N)}{rank(P_N)}= \frac{(d-1)! \alpha!}{(d-1+|\alpha|)!}
\end{align*}
\end{proof}
\begin{cor}
    For a Toeplitz-like operator $T \in \mathcal{T}_d$ of the form
    \begin{align*}
      T = \sum_{i=1}^n p_i(S)^*q_i(S) + K
    \end{align*}
    where $K$ is a compact operator on $H_d^2$, we have,
    \begin{align*}
        \lim_{N \rightarrow \infty}\chi_N (T) = \int_{\partial \mathbb{B}_d} \sum_{i=1}^n \overline{p_i(z)} q_i(z)d\sigma
    \end{align*}
    where $\sigma$ is the normalized Lebesgue measure on $\partial \mathbb{B}_d$.
\end{cor}
\begin{proof}
By Propositions $\ref{eq17}, \ref{eq18}$ and Proposition 1.4.9 in \cite{MR601594}
\begin{align*}
    \lim_{N \rightarrow \infty} \chi_N(S^{\alpha*} S^{\alpha})= \frac{(d-1)! \alpha!}{(d-1+|\alpha|)!} =\int_{\partial \mathbb{B}_d} \overline{z^{\alpha}} z^{\alpha} d\sigma
\end{align*}
By Proposition 1.4.8 in \cite{MR601594}, for multi-indices $\alpha$ and $\beta$ where $\alpha \neq \beta$,
\begin{align*}
    \chi_N(S^{\alpha *}S^{\beta}) = 0 = \int_{\partial \mathbb{B}_d} \overline{z^{\beta}} z^{\alpha} d\sigma
\end{align*}
Since we also have by Proposition $\ref{eq19}$, that $\lim_{N \rightarrow \infty}\chi_N(K)=0,$ it follows that,
\begin{align*}
        \lim_{N \rightarrow \infty}\chi_N (T) = \int_{\partial \mathbb{B}_d} \sum_{i,j} \overline{p_i(z)} q_j(z)d\sigma
    \end{align*}
    where $\sigma$ is the Lebesgue measure on $\partial \mathbb{B}_d$.
\end{proof}
\begin{proposition}
   For a Toeplitz-like operator $T$ with a continuous symbol $\pi(T)=\phi$,
   the limit of the normalized trace exists, and 
   \begin{align*}
   \lim _{N \rightarrow \infty} \chi_N(T) = \int_{\partial \mathbb{B}_d} \phi(z,\bar{z})d\sigma
    \end{align*}
   %The tracial state $\chi$ on $C(\partial \mathbb{B}_d) \simeq \mathcal{T}/K(H_d^2)$ defined as
    %\begin{align*}
    %\chi(f) = \int_{\partial \mathbb{B}_d} f d\sigma
    %\end{align*}
    %where $d\sigma$ is the surface area measure on $\partial \mathbb{B}_d$ extends to the whole of $\mathcal{T}_d$ and this extension agrees with
    %\begin{align*}
     %   \chi(T) : = \lim_{n \rightarrow \infty} \chi_n(T).
    %\end{align*}
    %on $\mathcal{T}_d$.
\end{proposition}
\begin{proof}

Finite sums of the form 
\begin{align}\label{eq1}
 A = \sum_{i=1}^n p_i(S)^* q_i(S) + K
\end{align}
are dense in $\mathcal{T}_d$. For $A$, we have by the previous proposition,
\begin{align*}
\lim_{N \rightarrow \infty} \chi_N(A) = \int \sum_{i=1}^n \overline{p_i(z)} q_i(z) d\sigma
\end{align*}
Now for an arbitrary element $T$ in the $C^*$-algebra $\mathcal{T}_d$ and $\epsilon >0$ choose an element $A$ of the form as in $(\ref{eq1})$ such that $||T-A|| < \epsilon$. Choose $N'$ such that for all $M,N >N'$, $|\chi_M(A)- \chi_N(A)| < \epsilon$. Now we note that,
\begin{align*}
|\chi_N(T) - \chi_M(T)| & = |\chi_N(T) -\chi_N(A) + \chi_N(A) -\chi_M(A) +\chi_M(A) -\chi_M(T)| \\
& \leq |\chi_N (T)- \chi_N(A)| + |\chi_M(A) -\chi_N(A)| + |\chi_M(A) - \chi_M(T)|\\
& \leq ||\chi_N|| \ ||T-A|| + \epsilon + ||\chi_M|| \ ||T-A||\\
& \leq 2 ||T-A|| + \epsilon\\
& \leq 3 \epsilon
\end{align*}
for all $M,N > N'$. Thus, the sequence $\chi_N(T)$ converges for each $T \in \mathcal{T}_d$.\\ Let us call this limit $\chi(T)$. It follows that $\chi$ is also a state and $\chi$ is the unique $weak^*$ limit of $\chi_N$.\\

Since $\chi$ is zero on $K(H_d^2)$ it induces a state on the $C^*$-algebra $\mathcal{T}_d/K(H_d^2)$ which we continue to denote by $\chi$. We have proved so far that $\chi$ agrees with the linear functional $f \mapsto \int_{\partial \mathbb{B}_d}f d\sigma$ on $C(\partial \mathbb{B}_d)$ on the dense set of polynomials of the form ($\ref{eq1}$). Hence, by continuity, they agree everywhere and the result follows.

%Note that for an element $T_p +K $ of $\mathcal{T}_d/ K(H_d^2)$ with a polynomial symbol $p$, we have by the previous proposition that,
%\begin{align*}
%\lim_{n \rightarrow \infty} \chi_n(T_p+K) = \int_{\partial \mathbb{B}_d} p d \sigma
%\end{align*}
%Thus for $T_h \in \mathcal{T}_d+ K$ with continuous symbol $h$, we choose a sequence of polynomials $p_n$ such that $T_{p_n}+K$ converges to $T_h+K$ (this is because of the isomorchism in (2)). Thus we have,
%\begin{align*}
%\lim_{n \rightarrow \infty} \chi_n(T_h+K) = \int_{\partial \mathbb{B}_d} h d \sigma
%\end{align*}
%Since we know by Prop 1.1 $\chi(K) = 0$ for a compact operator
\end{proof}

%\begin{proof}\label{eq10}
 %   Let $T$ be a positive Toeplitz operator with continuous symbol $f$ and let $[m,M]$ be the essential range of $f$. Suppose $\lambda_1, \lambda_2,\cdots, \lambda_{d_n}$ are the $d_n$ eigenvalues of $P_n TP_n$ where $d_n= rank(P_n)$. For a continuous function $g$ on $[m,M]$, we have,
  %  \begin{align*}
   % \lim_{n \rightarrow \infty}\int_m^M g d \mu_{n} = \int_{\partial \mathbb{B}_n} g \circ f d \sigma
   % \end{align*}
   % where $\sigma$ is the normalized Lebesgue measure on  $\partial \mathbb{B}_n$.
%\end{proof}

\begin{proof}[Proof of Theorem 1.1]
Since $\chi(T) = \lim_{N\rightarrow \infty} \frac{Tr(P_N T P_N)}{rank(P_N)}$ for all $T \in \mathcal{T}_d$ and $\{P_N\}$ is a Følner sequence, by Theorem $\ref{eq20}$, $\{\{P_N\}, \chi\}$ forms a Szeg\H{o} pair. 
\end{proof}
\begin{cor}
Let $T$ be a Toeplitz-like operator $T$ with a continuous symbol $\phi$. Then,
    \begin{align}\label{eq12}
    \lim_{N \rightarrow \infty} det(T_N)^{\frac{1}{N}} = \exp\Big( \int_{\mathbb{T}} log(\phi) \  dm \Big)
\end{align}
\end{cor}
We now indicate how Theorem $\ref{eq10}$ can be generalized to a larger family of functions on the unit ball.
    For $\alpha > -1$, the weighted Lebesgue measure $d \nu_{\alpha}$ is defined by
    \begin{align*}
        d \nu_{\alpha}(z) = c_{\alpha} (1-|z|^2)^{\alpha} d \nu(z)
    \end{align*}
    where $c_{\alpha} = \frac{\Gamma(d+\alpha+1)}{d! \Gamma(\alpha +1)}$ is such that $d\nu_{\alpha}$ is a probability measure on $\mathbb{B}_d$. For $\alpha > -1, p> 0$ the weighted Bergman space $A_p^{\alpha}$ consists of holomorchic functions $f$ in $L^p(\mathbb{B}_d, d \nu_{\alpha})$, that is,
    \begin{align*}
        A^p_{\alpha} = L^p(\mathbb{B}_d, d \nu_{\alpha}) \cap \mathcal{H}(\mathbb{B}_d).
    \end{align*}
    For $p=2$, $A^p_{\alpha}$ is a RKHS with kernel given by,
    \begin{align*}
        K^{\alpha}(z,w) = \frac{1}{(1-\langle z,w\rangle)^{d+1+\alpha}} \ \ z,w \in \mathbb{B}_d
    \end{align*}
    The functions,
    \begin{align*}
        e_m^{\alpha}(z) = \sqrt{\frac{\Gamma(d+|m|+\alpha+1)}{m! \Gamma(d+\alpha+1)}} z^m
    \end{align*}
    form an orthonormal basis for $A_{\alpha}^2$, where $m=(m_1,m_2,\cdots,m_d)$ runs over all multi-indices of non-negative integers.\\

    The shift operators $S^{\alpha}_i, i \in {1,2,\cdots,d}$ can be defined on the orthonormal basis as:
    \begin{align*}
        S^{\alpha}_i(e_m^{\alpha}):= z_i e_m^{\alpha}
    \end{align*}
    As before, the Toeplitz algebra of $A^p_{\alpha}$ can be defined as $\tau^{\alpha}_d : = C^*(1, S^{\alpha})$ where $S^{\alpha}:=(S^{\alpha}_1,S^{\alpha}_2,\cdots,S^{\alpha}_d)$.\\
    
    The matrix entries of $S^{\alpha}_i$ wrt. the orthonormal basis $e_m$ can be calculated as:
\begin{align*}
   S^{\alpha}_i e_m^{\alpha}(z)= \sqrt{\frac{m_i+1}{d+|m|+\alpha+2}} e_{m'}^{\alpha}
   \end{align*}
   where $m'=(m_1,m_2,\cdots,m_{i}+1,\cdots,m_d)$.\\
   Compare this with those of $S_i$:
   \begin{align*}
       S_i(e_m) =  \sqrt{\frac{m_i+1}{|m|+1}} e_{m'}
   \end{align*}
   where $m'=(m_1,m_2,\cdots,m_{i}+1,\cdots,m_d)$.

The ratio of these coefficients converges to $1$ as the size $|m|$ goes to infinity. i.e,
\begin{align*}
    \lim_{|m| \rightarrow \infty} \frac{ \sqrt{\frac{m_i+1}{|m|+1}}} {\sqrt{\frac{m_i+1}{d+|m|+\alpha+2}}} =1.
\end{align*}
Consequently, for each $i\in \{1,2,\cdots,d\},$ $S_i^{\alpha}$ is unitarily equivalent to $S_i(I+K_i)$ for some compact operators $K_i$. In other words, the Bergman shifts are unitarily equivalent to compact perturbations of the Drury-Arveson shifts. Since $\chi(K)=0$ for compact operators in our theorem, a similar argument can also be run for the Bergman Toeplitz-algebra. Thus, Theorem $\ref{eq10}$ holds for the weighted Bergman spaces $A^{\alpha}_2$ as well.
%\begin{remark}
 %   The conclusion in Corollary 1.18 holds true even if we choose the sequence $\{P_n\}$ to be projections onto
  %  \begin{align*}
   % H_n:= span\{e_{w_i}: 1\leq i\leq n\}
   % \end{align*}
   % where $e_{w_i}$ is the canonical orthonormal basis of $H_d^{2}$ in the lexicograchic order. In other words, $P_n$ are projections onto the span of first $n$ basis elements. In this case we have,
   % \begin{align*}
    %    Tr (S^{\alpha *} S^{\alpha}P_n) = \sum_{0\leq i\leq n} \frac{||S^{\alpha} e_{w_i}||^2}{||P_n||_1}
   % \end{align*}
% \end{remark}
%\begin{theorem}
 %$(\{P_{\alpha}\},\chi)$ is a Szeg\H{o} pair for the $C^*$-algebra $\mathcal{T}_d$. The weak limit of the eigenvalue distributions of ${T_f}_n := P_n T_{f} P_n$ is the pull back of the Lebesgue measure on $\partial \mathbb{B}_n$ under $f$.
%\end{theorem}
\printbibliography
\end{document}